\documentclass[a4paper]{amsart}                

\usepackage[latin1]{inputenc}                  
\usepackage[T1]{fontenc}                       
\usepackage[spanish,english]{babel}            
\usepackage{graphicx}                  
\usepackage{amsmath,amssymb,amsthm}            
\usepackage{latexsym}                          
\usepackage{delarray}                          
\usepackage{bbm}                               
\usepackage{hyperref}                          
\usepackage[pdftex,usenames,dvipsnames]{color} 
\usepackage{calrsfs,eufrak,mathrsfs,upgreek}   
\usepackage{bbding,trfsigns}                   
\usepackage{wasysym}                           
\usepackage{datetime}                          
\usepackage{color}                             

\DeclareMathAlphabet{\mathpzc}{OT1}{pzc}{m}{it} 

\newtheorem{teo}{Theorem}[section]
\newtheorem{propo}[teo]{Proposition}
\newtheorem{coro}[teo]{Corollary}

\DeclareMathAlphabet{\mathpzc}{OT1}{pzc}{m}{it}

\DeclareFontFamily{U}{mathx}{\hyphenchar\font45}
\DeclareFontShape{U}{mathx}{m}{n}{
      <5> <6> <7> <8> <9> <10>
      <10.95> <12> <14.4> <17.28> <20.74> <24.88>
      mathx10
      }{}
\DeclareSymbolFont{mathx}{U}{mathx}{m}{n}
\DeclareFontSubstitution{U}{mathx}{m}{n}
\DeclareMathAccent{\widecheck}{0}{mathx}{"71}
\DeclareMathAccent{\wideparen}{0}{mathx}{"75}

\allowdisplaybreaks

\author[V. Almeida]{V\'{\i}ctor Almeida$^1$}
\author[J.J. Betancor]{Jorge J. Betancor$^1$}
\author[L. Rodr\'iguez-Mesa]{Lourdes Rodr\'iguez-Mesa$^1$}

\address{$^1$
        Departamento de An\'alisis Matem\'atico,
        Universidad de La Laguna,
        Campus de Anchieta, Avda. Astrof\'{\i}sico Francisco S\'anchez, s/n,
        38271, La Laguna (Sta. Cruz de Tenerife), Spain.}
\email{valmeida@ull.es, jbetanco@ull.es, lrguez@ull.es}

\thanks{This paper is partially supported by MTM2013-44357-P}

\date{\today}

\begin{document}

\title[Molecules associated to Hardy spaces with pointwise variable anisotropy]
{Molecules associated to Hardy spaces with pointwise variable anisotropy}



\begin{abstract}
In this paper we introduce molecules associated to Hardy spaces with pointwise variable anisotropy, and prove that each molecule can be represented as a sum of atoms.
\end{abstract}

\maketitle

\section{Introduction}\label{sec:intro}
Bownik \cite{Bow} introduced and investigated anisotropic Hardy spaces associated with expansive matrices. Bownik's anisotropic Hardy spaces have been also studied in \cite{BB}, \cite{BD}, \cite{BLYZ}, \cite{Hu}, \cite{Wan} and \cite{ZL}, amongst others.

Recently, Dekel, Petrushev and Weissblat \cite{DPW} generalized Bownik's spaces by constructing Hardy spaces $H^p(\Theta)$, $0<p\leq 1$, associated with ellipsoid multilevel covers $ \Theta$. The anisotropy in Bownik's spaces is the same one in each point in $\mathbb{R}^n$, while anisotropy in $H^p(\Theta )$ can change rapidly from point to point in $\mathbb{R}^n$ and in depth from level to level. The ellipsoid covers $\Theta$ induce quasidistances on $\mathbb{R}^n$ and spaces of homogeneous type arise when on $\mathbb{R}^n$ we consider those quasidistances and the Lebesgue measure. Hardy spaces on spaces of homogeneous type have been studied by many authors (see, for instance, \cite{CW1}, \cite{CW2}, \cite{GLY}, \cite{HMY} and \cite{MS}). In the general setting of spaces of homogeneous type the Hardy space $H^p$ has not sufficient structure when $p$ is close to zero. However, anisotropic Hardy spaces defined in \cite{Bow} and the more general case one introduced in \cite{DPW} cover the full range $0<p\leq 1$.

In \cite{DPW}, after defining Hardy spaces $H^p(\Theta)$ by using maximal functions, atomic characterizations of the distributions in $H^p(\Theta )$ are established. The dual spaces of $H^p(\Theta )$ is described as a Campanato type space in \cite{DW}. Our objective in this paper is to introduce molecules associated with covers $\Theta$ showing that they can be represented as a sum of atoms in $H^p(\Theta )$. As it is well-known, molecular representations play an important role to study the boundedness of operators in Hardy spaces.

We now recall definitions and main results about variable anisotropy and Hardy spaces in this setting that can be found in \cite{DDP} and \cite{DPW}.

By $B$ we denote the Euclidean unit ball in $\mathbb{R}^n$. By an affine transformation $A$ in $\mathbb{R}^n$ we mean that one having the form $Ax=Mx+v$, $x\in \mathbb{R}^n$, where $M$ is a nonsingular $n\times n$ matrix and $v\in \mathbb{R}^n$. The set $\theta =A(B)$ is named the ellipsoid associated with $A$, and $v=A(0)$ is said to be the center of $\theta$.

Next we define continuous multilevel ellipsoid covers of $\mathbb{R}^n$. For every $t\in \mathbb{R}$, $\Theta _t$ is a set of ellipsoids (of $t$-level) satisfying the following conditions:

$(C1)$ For every $x\in \mathbb{R}^n$ there exist a unique ellipsoid $\theta (x,t)\in \Theta _t$ and an affine transform $A_{x,t}$ defined by $A_{x,t}(y)=M_{x,t}y+x$, $y\in \mathbb{R}^n$, such that $\theta (x,t)=A_{x,t}(B)$ and
$$
a_12^{-t}\leq |\theta (x,t)|\leq a_22^{-t}.
$$

$(C2)$ For every $x,y\in \mathbb{R}^n$, $t\in \mathbb{R}$ and $s\geq 0$, if $\theta (x,t)\cap \theta (y,t+s)\not =\emptyset$, then
$$
a_32^{-a_4s}\leq \|M_{y,t+s}^{-1}M_{x,t}\|^{-1}\leq \|M_{x,t}^{-1}M_{y,t+s}\|\leq a_52^{-a_6s}.
$$
We say that the set $\Theta =\cup_{t\in \mathbb{R}}\Theta _t$ is a continuous multilevel ellipsoid cover of $\mathbb{R}^n$.
Note that the set of parameters, which is denoted by $p(\Theta )=\{a_1,a_2,a_3,a_4,a_5,a_6\}$, does not depend on $x\in \mathbb{R}^n$ and $t\in \mathbb{R}$.

A careful reading of the proof of \cite[Lemma 2.2]{DPW} and minor modifications in it, allow us to establish the following useful property.
\begin{propo}\label{Prop3}
Assume that $\Theta $ is a continuous multilevel ellipsoid cover of $\mathbb{R}^n$. There exists $J\geq (\log_2a_5)/a_6$ such that, for every $x\in \mathbb{R}^n$, $t\in\mathbb{R}$, and $s\geq J$,
$$
\theta (x,t)\subset \theta (x,t-s).
$$
\end{propo}

A quasidistance on a set $X$ is a mapping $\rho :X\times X\longrightarrow [0,\infty )$ satisfying the following conditions:

$(a)$ $\rho (x,y)=0\iff x=y$;

$(b)$ $\rho (x,y)=\rho (y,x)$, $x,y\in X$;

$(c)$ for some $K\geq 1$, $\rho (x,y)\leq K(\rho (x,z)+\rho (z,y))$, $x,y,z\in X$.

The ellipsoid cover $\Theta$ induces a quasidistance $\rho _\Theta$ in $\mathbb{R}^n$ defined as follows:
$$
\rho _\Theta (x,y)=\inf\Big\{|\theta |:  \theta\in \Theta \mbox{ and }x,y\in\theta \Big\},\quad x,y\in \mathbb{R}^n.
$$

As usual $S(\mathbb{R}^n)$ denotes the Schwartz function space on $\mathbb{R}^n$ and $S'(\mathbb{R}^n)$ represents its dual space.

If $\alpha =(\alpha _1,...,\alpha _n)\in \mathbb{N}^n$, we write $|\alpha|=\alpha _1+...+\alpha _n$. Suppose that $N,\widetilde{N}\in \mathbb{N}$, $N\leq \widetilde{N}$, $\alpha \in \mathbb{N}^n$, $|\alpha |\leq N$, and $\psi \in C^N(\mathbb{R}^n)$. We define
$$
\|\psi \|_{\alpha ,\widetilde{N}}=\sup _{y\in \mathbb{R}^n}(1+|y|)^{\widetilde{N}}|\partial ^\alpha \psi (y)|,
$$
where $\partial ^\alpha =\frac{\partial ^{|\alpha |}}{\partial x_1^{\alpha _1}...\partial x_n^{\alpha _n}}$, and
$$
\|\psi \|_{N,\widetilde{N}}=\max_{|\alpha|\leq N}\|\psi \|_{\alpha ,\widetilde{N}}.
$$
Also we consider
$$
S_{N,\widetilde{N}}=\Big\{\psi \in S(\mathbb{R}^n): \|\psi \|_{N,\widetilde{N}}\leq 1\Big\}.
$$

Let $\Theta$ be a continuous multilevel ellipsoid cover of $\mathbb{R}^n$, where $\theta (x,t)=x+M_{x,t}(B)$, for every $x\in \mathbb{R}^n$ and $t\in \mathbb{R}$. We denote by $\psi _{x,t}$, $x\in \mathbb{R}^n$ and $t\in \mathbb{R}$, the function given by
$$
\psi _{x,t}(y)=|{\rm det} (M_{x,t}^{-1})|\psi (M_{x,t}^{-1}(x-y)),\quad y\in \mathbb{R}^n.
$$
If $g\in S'(\mathbb{R}^n)$ and $\psi \in S(\mathbb{R}^n)$, the radial maximal function $M_\psi (g)$ of $g$ associated with $\Theta$ is defined by
$$
M_\psi (g)(x)=\sup_{t\in \mathbb{R}}|\langle g(y),\psi _{x,t}(y)\rangle |,\quad x\in \mathbb{R}^n,
$$
and for every $0<N\leq \widetilde{N}$, $N,\widetilde{N}\in \mathbb{N}$, the radial grandmaximal function $M_{N,\widetilde{N}}(g)$ of $g$ associated with $\Theta$ is given as follows:
$$
M_{N,\widetilde{N}}(g)(x)=\sup_{\psi \in S_{N,\widetilde{N}}}M_\psi (g)(x),\quad x\in \mathbb{R}^n.
$$

Given the set of parameters $p(\Theta)$, for every $0<p\leq 1$, we define
$$
N_p(\Theta )=\min\Big\{m\in \mathbb{N}: m>(\max\{1,a_4\}n+1) /(a_6p)\Big\},
$$
and
$$
\widetilde{N_p}(\Theta )=\min\Big\{m\in \mathbb{N}: m>(a_4N_p(\Theta )+1) /a_6\Big\}.
$$

Let $0<p\leq 1$. The Hardy space $H^p(\Theta )$ associated with $\Theta$ consists of all those distributions $g\in S'(\mathbb{R}^n)$ such that $M_{N_p,\widetilde{N_p}}(g)\in L^p(\mathbb{R}^n)$. $H^p(\Theta)$ is endowed with the quasinorm $\|\cdot \|_{H^p(\Theta )}$ defined by
$$
\|g\|_{H^p(\Theta )}=\|M_{N_p,\widetilde{N_p}}(g)\|_{L^p(\mathbb{R}^n)},\quad g\in H^p(\Theta).
$$

In \cite[Section 4]{DPW} Hardy spaces $H^p(\Theta )$ were characterized by using atomic decompositions.

Let the triplet $(p,q,m)$ admissible for $\Theta$, that is, $0<p\leq 1$, $1\leq q\leq \infty$, $p<q$, $m \in \mathbb{N}$ and $m \geq N_p(\Theta )$. A measurable function ${\mathfrak a}$ on $\mathbb{R}^n$ is a $(p,q,m)$-atom associated with $x\in \mathbb{R}^n$ and $t\in \mathbb{R}$, when the following properties are satisfied:

$(i)$ ${\rm supp }({\mathfrak a})\subset \theta (x,t)\in \Theta$.

$(ii)$ $\|{\mathfrak a}\|_{L^q(\mathbb{R}^n)}\leq |\theta (x,t)|^{1/q-1/p}$, where $1/q$ is understood as $0$ when $q=\infty$.

$(iii)$ $\int {\mathfrak a}(y)y^\alpha dy=0$, when $\alpha \in \mathbb{N}^n$ and $|\alpha |\leq m$. Here, if $\alpha =(\alpha _1,...,\alpha _n)\in \mathbb{N}^n$ and $y=(y_1,...,y_n)\in \mathbb{R}^n$, $y^\alpha =y_1^{\alpha _1}...y_n^{\alpha _n}$.

If $(p,q,m)$ is admissible for $\Theta$, the atomic Hardy space $H_{q,m }^p(\Theta )$ consists of all those distributions $g\in S'(\mathbb{R}^n)$ that can be written as $g=\sum_{j\in \mathbb{N}}\lambda _j{\mathfrak a}_j$, where the series converges in $S'(\mathbb{R}^n)$, $\{{\mathfrak a}_j\}_{j\in \mathbb{N}}$ is a sequence of $(p,q,m)$-atoms and $\{\lambda _j\}_{j\in \mathbb{N}}$ is a sequence of positive numbers such that $\sum_{j\in \mathbb{N}}\lambda _j^p<\infty $. On $H_{q,m}^p(\Theta )$ we consider the quasinorm $\|\cdot \|_{H_{q,m}^p(\Theta )}$ defined by
$$
\|g\|_{H^p_{q,m}}(\Theta )=\inf \Big(\sum_{j\in \mathbb{N}}\lambda _j^p\Big)^{1/p},
$$
where the infimum is taken over all the sequences $\{\lambda _j\}_{j\in  \mathbb{N}}$ of positive numbers such that $\sum_{j\in \mathbb{N}}\lambda _j^p<\infty $ and $g=\sum_{j\in \mathbb{N}}\lambda _j{\mathfrak a}_j$, in the sense of convergence in $S'(\mathbb{R}^n)$, for a certain sequence $\{{\mathfrak a}_j\}_{j\in \mathbb{N}}$ of $(p,q,m)$-atoms.

In \cite[Section 4]{DPW} it was proved that $H^p(\Theta )=H_{q,m}^p(\Theta )$ where the equality is understood algebraic and topologically.

In order to establish that an operator (for instance, a singular integral operator) is bounded from a Hardy space into a Lebesgue space, many times it is sufficient to use the atomic characterization of the Hardy space (see for instance \cite{Bown} and \cite{BLYZ}). However, in general, it is not true that if $T$ is a bounded operator from a Hardy space $H^p$ into $L^p$ (you can think $T$ being the Hilbert transform on ${\Bbb R}$) and ${\mathfrak a}$ is an atom in $H^p$, then $T\mathfrak{a}$ is a finite linear combination of atoms in $H^p$. Then, in order to prove the boundedness of an operator between Hardy spaces, a class of functions larger than atoms is needed. These other functions are called molecules and suitably controlled infinite linear combination of molecules are in the considered Hardy space.

Taibleson and Weiss (\cite{TW1} and \cite{TW2}) established the first results about molecular decomposition for distributions in Hardy spaces. For weighted Hardy spaces, molecular characterizations were obtain in \cite{LL} and \cite {LP}. Recently, molecular representations of distributions in the Bownik anisotropic Hardy spaces have been established (\cite{Wan}, in the unweighted case, and \cite {ZL}, in the weighted case).

We now introduce molecules in our particular anisotropic setting.

Let $\Theta$ be a continuous multilevel ellipsoid cover of $\mathbb{R}^n$ and $p(\Theta)$ the associated set of parameters. Consider $0<p\leq 1\leq q$, $p<q$, $m\in \mathbb{N}$, and
$$
d>\max\Big \{a_4\big(m+n-\frac{n}{q}\big), \frac{a_4}{a_6}\big(1-\frac{1}{q}+ma_4\big), \frac{a_4}{a_6}\big(\frac{1}{p}-\frac{1}{q}+m(a_4-a_6)\big)\Big\}.
$$

We say that a function $b\in L^q(\mathbb{R}^n)$ is a $(p,q,m,d)$-molecule centered in $x_0\in \mathbb{R}^n$ when the following properties hold:

$(M1)$ $\int b(x)x^\alpha dx=0$, for every $\alpha \in \mathbb{N}^n$, $|\alpha |\leq m$.

$(M2)$ $b\rho (x_0,\cdot  )^d\in L^q(\mathbb{R}^n)$.

For such functions $b$, we define the $(p,q,m,d)$-molecular "norm" $M_{p,q,m,d}(b)$ by
$$
M_{p,q,m,d}(b)=\big\|b\rho ( x_0,\cdot)^d\big\|_{L^q(\mathbb{R}^n)}^\sigma \Big(\big\|b\big\|_{L^q(\mathbb{R}^n)}^{(1-\sigma )\alpha _1}+\big\|b\big\|_{L^q(\mathbb{R}^n)}^{(1-\sigma )\alpha _2}\Big),
$$
where
\begin{align*}
&\sigma =\frac{1}{d}\Big(\frac{1}{p}-\frac{1}{q}\Big),&\\
&\alpha _1=\frac{1}{d(1-\sigma)}\Big(\frac{1}{q}-\frac{1}{p}+\frac{da_6}{a_4}-m(a_4-a_6)\Big),\\
&\alpha _2=\frac{1}{d(1-\sigma )}\Big(\frac{1}{q}-\frac{1}{p}+\frac{da_4}{a_6}+m(a_4-a_6)\Big).
\end{align*}

Note that in the isotropic case $a_4=a_6=1$, and then, $\alpha _1=\alpha _2=1$. In this case our molecules and the molecular "norm" are the classical ones.

In order to describe our molecules as suitable sums of multiple of atoms in $H^p(\Theta)$, we need that the ellipsoids of zero-level are equivalent in shape. We say that a continuous multilevel ellipsoid cover of $\mathbb{R}^n$ is quasi zero uniform, when, there exist positive constants $c_1$ and $c_2$ such that
$$
c_1\leq \big\|M_\theta ^{-1}M_{\theta '}\big\|\leq c_2, \quad \theta, \theta '\in \Theta_0.
$$
Here $M_\theta $ and $M_{\theta '}$ have the obvious mean. Quasi zero uniform covers of ${\Bbb R}^n$ were also considered in \cite{De}.

Our main result which will be proved in next section is the following.

\begin{teo}\label{Th1}
Let $\Theta$ be a quasi zero uniform continuous multilevel ellipsoid cover of $\mathbb{R}^n$. Assume that $(p,q,m)$ is admissible for $\Theta$ and that
$$
d>\max\Big \{a_4\big(m+n-\frac{n}{q}\big), \frac{a_4}{a_6}\big(1-\frac{1}{q}+ma_4\big), \frac{a_4}{a_6}\big(\frac{1}{p}-\frac{1}{q}+m(a_4-a_6)\big)\Big\}.
$$
Then, for every $(p,q,m,d)$-molecule $b$, there exist a sequence $\{{\mathfrak a}_j\}_{j\in\Bbb N}$ of $(p,q,m)$-atoms and a sequence $\{\lambda_j\}_{j\in\Bbb N}\subset (0,\infty)$ with $\sum_{j\in\Bbb N}\lambda_j^p\leq CM_{p,q,m,d}(b)$, where $C>0$ does not depend on $b$, verifying that
$$b=\sum_{j\in\Bbb N}\lambda_j{\mathfrak a}_j, \;\;\;\;\mbox{in}\;S'({\Bbb R}^n).$$
\end{teo}

As an immediate consequence of Theorem \ref{Th1} and the fact that $H_{q,m}^p(\Theta )\subset H^p(\Theta )$ we obtain the following result.

\begin{coro}\label{coro1.3}
Let $\Theta$ be a quasi zero uniform continuous multilevel ellipsoid cover of $\mathbb{R}^n$ and let $(p,q,m)$ be admissible for $\Theta$. Assume that
$$
d>\max\Big \{a_4\big(m+n-\frac{n}{q}\big), \frac{a_4}{a_6}\big(1-\frac{1}{q}+ma_4\big), \frac{a_4}{a_6}\big(\frac{1}{p}-\frac{1}{q}+m(a_4-a_6)\big)\Big\}.
$$
and that $(b_j)_{j\in\Bbb N}$ is a sequence of $(p,q,m,d)$-molecules such that, for a certain $C>0$, $M_{p,q,m,d}(b_j)\leq C$, $j\in\Bbb N$, and that  $\{\lambda_j\}_{j\in\Bbb N}$ is a sequence of complex number satisfying that $\sum_{j\in\Bbb N}|\lambda_j|^p<\infty$. Then, the series $\sum_{j\in\Bbb N}\lambda_jb_j$ converges in $H^p(\Theta )$ and
$\|\sum_{j\in\Bbb N}\lambda_jb_j\|_{H^p(\Theta )}\leq C(\sum_{j\in\Bbb N}|\lambda_j|^p)^{1/p}$.
\end{coro}

Throughout this paper by $C$ we always denote a positive constant that can change in each ocurrence.

\section{Proof of Theorem \ref{Th1}}
Let us denote by $\lambda$ and $\widetilde{\lambda}$ the functions on $\mathbb{R}$ defined by
$$
\lambda (t)=\left\{\begin{array}{ll}
			a_6t,& t\geq 0,\\
			a_4t,&t<0,
		         \end{array}
\right.
\quad \mbox{ and }\quad
\widetilde{\lambda} (t)=\left\{\begin{array}{ll}
			a_4t,& t\geq 0,\\
			a_6t,&t<0.
		         \end{array}
\right.
$$
The following results concerning quasi zero uniform continuous multilevel ellipsoid covers of $\mathbb{R}^n$ will be very useful.

\begin{propo}\label{Prop1}
Assume that $\Theta $ is a quasi zero uniform continuous multilevel ellipsoid cover of $\mathbb{R}^n$. Let $x_0\in \mathbb{R}^n$ and $t_0\in \mathbb{R}$, and $\theta (x_0,t_0)\in  \Theta$. Then, there exists $C>0$ such that

(a) $|z-x|\leq C2^{-\lambda (t_0)}$, $z,x\in \theta (x_0,t_0)$, and

(b)  $|z-x_0|\geq C2^{-\widetilde{\lambda }(t_0)}$, $z\not\in \theta (x_0,t_0)$.
\end{propo}

\begin{proof}
As in \cite{De}, if $\Omega $ is a convex bounded domain in $\mathbb{R}^n$ we denote by ${\rm width }(\Omega)$ the diameter of the largest $n$-dimensional Euclidean ball that is contained in $\Omega$.

In the proof of \cite[Theorem 4.8]{De} it was established that if $\theta \in \Theta _t$ with $t\geq 0$, then
$$
\frac{2^{-ta_4}}{C}\leq {\rm width }(\theta )\leq {\rm diam }(\theta )\leq C2^{-ta_6},
$$
for a certain $C>0$.

Suppose now that $\theta \in \Theta _t$ with $t<0$. We choose $\theta '\in \Theta _0$ such that $\theta '\cap \theta \not =\emptyset$. According to the property $(C2)$ and by using that $\Theta$ is quasi zero uniform we deduce that, for certain $C>0$,
$$
{\rm diam}(\theta)\leq 2\|M_\theta\|=2\|M_{\theta '}M_{\theta '}^{-1}M_\theta\|\leq 2\|M_{\theta '}\|\|M_{\theta '}^{-1}M_\theta\|\leq C2^{-a_4t},
$$
and
$$
{\rm width }(\theta)\geq \frac{1}{2}\|M_\theta ^{-1}\|^{-1}=\frac{1}{2}\|M_\theta ^{-1}M_{\theta '}M_{\theta '}^{-1}\|^{-1}\geq \frac{1}{2}\|M_\theta ^{-1}M_{\theta '}\|^{-1}\|M_{\theta '}^{-1}\|^{-1}\geq \frac{2^{-a_6t}}{C}.
$$

In order to see $(a)$ it is sufficient to note that $|z-x|\leq {\rm diam} (\theta (x_0,t_0))$, provided that $z,x\in \theta (x_0,t_0)$. Property $(b)$ follows because $|z-x_0|\geq {\rm width} (\theta (x_0,t_0))$ when $z\not \in \theta (x_0,t_0)$.
\end{proof}

\begin{propo}\label{Prop2}
Assume that $\Theta$ is a quasi zero uniform continuous multilevel ellipsoid cover of $\mathbb{R}^n$. Then, there exists $C>0$ such that, for every $x\in \mathbb{R}^n$,
$$
\|M_{x,t}\|\leq C2^{-\lambda (t)}
\quad \mbox{ and }\quad
\|M_{x,t}^{-1}\|\leq C2^{\widetilde{\lambda }(t)}.
$$
\end{propo}

\begin{proof}
These results can be proved by using the property $(C2)$ and that $\Theta$ is quasi zero uniform.
\end{proof}

An immediate consequence of Proposition \ref{Prop1} is the following corollary.

\begin{coro}\label{coro}
Let $\Theta$ be a quasi zero uniform continuous multilevel ellipsoid cover of $\mathbb{R}^n$. There exists $c_0>1$  such that
\begin{equation}\label{coro(a)}
\frac{2^{-\widetilde{\lambda }(s)}}{c_0}\leq |z-x|\leq c_02^{-\lambda (t)},
\end{equation}
and
\begin{equation}\label{coro(b)}
\rho (x,z)\geq a_12^{-\lambda ^{-1}(2\log_2c_0+\widetilde{\lambda}(s))},
\end{equation}
for every $x,y,z\in \mathbb{R}^n$ and $t,s\in \mathbb{R}$ verifying that $z,x\in \theta (y,t)$ and $z\not \in \theta (x,s)$.
\end{coro}
\begin{proof}
Estimation (\ref{coro(a)}) follows directly from Proposition \ref{Prop1}. To prove inequality (\ref{coro(b)}), let $x,z\in \mathbb{R}^n$ and $s\in \mathbb{R}$ such $z\not \in \theta (x,s)$. It is sufficient to use condition (C1) and to note that inequalities in (\ref{coro(a)}) lead to $2^{-\widetilde{\lambda }(s)}\leq c_0^22^{-\lambda (t)}$, that is,
$$
t\leq \lambda ^{-1}(2\log_2c_0+\widetilde{\lambda}(s)),
$$
for every $t\in \mathbb{R}$ for which there exists an ellipsoid $\theta \in \Theta _t$ such that $x,z\in \theta$.
\end{proof}

Next we establish a useful property about integrability for the molecules.

\begin{propo}\label{Prop4}
Suppose that $b$ is a $(p,q,m,d)$-molecule. Then,
$$
\int_{\mathbb{R}^n}|b(x)||x^\alpha |dx<\infty, \quad \alpha \in \mathbb{N}^n,\;|\alpha |\leq m.
$$
\end{propo}
\begin{proof}
Assume that $b$ is centered in $x_0\in \mathbb{R}^n$. According to \cite[Theorem 2.9]{DPW}, there exits  $C_0\geq 1$ such that
$|x-x_0|\leq C_0\rho (x_0,x)^{a_4}$, $x\in \mathbb{R}^n\mbox{ and }\rho (x_0,x)\geq 1$. Then,
$$
\rho (x_0,x)\geq (|x-x_0|/C_0)^{1/a_4}, \quad {\rm when }\;\;|x-x_0|\geq C_0.
$$

Let $\alpha \in \mathbb{N}^n$, with $|\alpha |\leq m$. It is clear that if $|x-x_0|\geq C_0$, then $|x^\alpha |\leq C_1|x-x_0|^m$, for a certain $C_1>0$. Thus, by using the properties of the molecules we deduce that
\begin{align*}
\int_{\mathbb{R}^n}|b(x)||x^\alpha |dx&=\left(\int_{B(x_0,C_0)}+\int_{(B(x_0,C_0))^c}\right)|b(x)||x^\alpha |dx\\
&\hspace{-2cm}\leq \|b\|_{L^q(\mathbb{R}^n)}\|x^\alpha \|_{L^{q'}(B(x_0,C_0))}+C_1\int_{(B(x_0,C_0))^c}|b(x)|\rho (x_0,x)^d\frac{|x-x_0|^m}{\rho (x_0,x)^d}dx\\
&\hspace{-2cm}\leq \|b\|_{L^q(\mathbb{R}^n)}\|x^\alpha \|_{L^{q'}(B(x_0,C_0))}+C\big\|b\rho (x_0,\cdot)^d\big\|_{L^{q}(\mathbb{R}^n)}\big\||x-x_0|^{m-d/a_4}\big\|_{L^{q'}((B(x_0,C_0))^c)}<\infty ,
\end{align*}
because $(m-d/a_4)q'+n<0$.
\end{proof}

\begin{proof}[\sc Proof of Theorem \ref{Th1}]
Suppose that $b$ is a $(p,q,m,d)$-molecule centered in $x_0\in \mathbb{R}^n$ such that $M_{p,q,m,d}(b)=1$.
We define $\gamma =\|b\|_{L^q(\mathbb{R}^n)}^{1/(1/q-1/p)}$ and consider $r\in \mathbb{Z}$ for which $2^{-r-1}<\gamma \leq 2^{-r}$.

According to Proposition \ref{Prop3}, we take $\ell \in \mathbb{N}$ (which does not depend on $x_0$) such that
\begin{equation}\label{contenido}
\theta (x_0,t)\subset \theta (x_0,t-s), \quad t\in \mathbb{R},\,s\geq \ell,
\end{equation}
and for any $j\in \mathbb{N}$ we consider the set $E_j=\theta (x_0,r-j\ell)$ and the function $b_j=b\chi _{E_j}$.

According to \cite[Lemma 2.2]{DDP}, if we denote by $\sigma_t$ the minimum semi-axes of the ellipsoid $\theta (x_0,t)$, there exist $c,C>0$ such that
$$
\sigma_{t_1}\geq C\sigma_{t_2}2^{c(t_2-t_1)}, \quad -\infty <t_1\leq t_2<+\infty.
$$
Then, $\cup_{j\in \mathbb{N}}E_j=\mathbb{R}^n$. Also, by (\ref{contenido}) we have that $E_j\subset E_{j+1}$, $j\in \mathbb{N}$.

Let $k\in \mathbb{N}$. We denote by $\mathcal{P}_{m,k}$ the linear space that consists of all the restrictions to $E_k$ of polynomials in $\mathbb{R}^n$ with degree less or equal than $m$. It is clear that $\mathcal{P}_{m,k}$ has $\{x^\alpha \chi _{E_k},\alpha \in \mathbb{N}^n,|\alpha |\leq m\}$ as a linear basis on $E_k$. According to \cite[Lemma 2.1]{Se} there exists a unique $P_{m,k}\in \mathcal{P}_{m,k}$ such that
$$
\int_{\mathbb{R}^n}(b_k(x)-P_{m,k}(x))x^\beta dx=0,\quad \beta \in \mathbb{N}^n,|\beta|\leq m.
$$
We can write
$$
P_{m,k}=\sum_{\alpha \in \mathbb{N}^n, |\alpha |\leq m}\frac{1}{|E_k|}\int_{\mathbb{R}^n}b_k(x)x^\alpha dx\;Q_{m,k;\alpha},
$$
where, for every $\alpha \in \mathbb{N}^n$, $|\alpha |\leq m$, $Q_{m,k;\alpha}$ is the unique element of $\mathcal{P}_{m,k}$ such that
$$
\int_{\mathbb{R}^n}Q_{m,k;\alpha }(x)x^\beta dx=|E_k|\delta _{\alpha ,\beta},\quad \beta \in \mathbb{N}^n,|\beta|\leq m.
$$
Here, as usual, $\delta _{\alpha ,\beta}$ denotes de Kronecker function.

Let $g\in L^1(B)$. We denote by $\Pi _{B,m}(g)$ the unique polynomial in $ \mathbb{R}^n$ with degree less or equal than $m$ such that
$$
\int_B\Pi _{B,m}(g)(x)x^\beta dx=\int_Bg(x)x^\beta dx,\quad \beta \in \mathbb{N}^n,|\beta|\leq m.
$$

We note that we can write
$$
\Pi_{B,m}(g)=\sum_{\alpha \in \mathbb{N}^n,|\alpha |\leq m}\frac{1}{|B|}\int_Bg(z)z^\alpha dz\;Q_\alpha,
$$
where, for every $\alpha \in \mathbb{N}^n$, $|\alpha |\leq m$, $Q_\alpha$ is a certain polynomial with degree less than or equal to $m$. Thus,
we can find $C_0>0$ which depends only on $m$ such that
\begin{equation}\label{Pi}
|\Pi_{B,m}(g)(x)|\leq \frac{C_0}{|B|}\int_B|g(z)|dz,\quad x\in B.
\end{equation}
If $M$ is a $n\times n$-matrix and $z\in \mathbb{R}^n$, we define $(D_Mg)(x)=g(Mx)$, $x\in \mathbb{R}^n$, and $(\tau _zg)(x)=g(x-z)$, $x\in \mathbb{R}^n$.
Since $E_k=x_0+M_{x_0,r-k\ell}(B)$,
\begin{equation}\label{M1}
P_{m,k}=\chi_{E_k}(\tau _{x_0}D_{M_{x_0,r-k\ell}^{-1}}\Pi_{B,m}D_{M_{x_0,r-k\ell }}\tau _{-x_0})(b_k).
\end{equation}
Indeed, let $\beta \in \mathbb{N}^n$, $|\beta|\leq m$. We have that
\begin{align*}
\int_{E_k}(\tau _{x_0}D_{M_{x_0,r-k\ell}^{-1}}\Pi_{B,m}D_{M_{x_0,r-k\ell }}\tau _{-x_0})(b_k)(x)x^\beta dx&\\
&\hspace{-4cm}=\int_{M_{x_0,r-k\ell }(B)}(D_{M_{x_0,r-k\ell}^{-1}}\Pi_{B,m}D_{M_{x_0,r-k\ell }}\tau _{-x_0})(b_k)(x)\tau_{-x_0}(x^\beta)dx\\
&\hspace{-4cm}=|E_k|\int_{B}(\Pi_{B,m}D_{M_{x_0,r-k\ell }}\tau _{-x_0})(b_k)(x)(D_{M_{x_0,r-k\ell}}\tau _{-x_0})(x^\beta)dx\\
&\hspace{-4cm}=|E_k|\int_{B}(D_{M_{x_0,r-k\ell }}\tau _{-x_0})(b_k)(x)(D_{M_{x_0,r-k\ell}}\tau _{-x_0})(x^\beta)dx\\
&\hspace{-4cm}=\int_{\mathbb{R}^n}b_k(x)x^\beta dx.
\end{align*}

Now, and by using (\ref{Pi}) and (\ref{M1}), we get
\begin{align}\label{Pi2}
|P_{m,k}(x)|&=|(\Pi _{B,m}D_{M_{x_0,r-k\ell}}\tau _{-x_0})(b_k)(M_{x_0,r-k\ell}^{-1}(x-x_0))|\leq \frac{C_0}{|B|}\int_B|b(x_0+M_{x_0,r-k\ell }y)|dy \nonumber\\
&=\frac{C_1}{|E_k|}\int_{E_k}|b(z)|dz,\quad x\in E_k,
\end{align}
where $C_1=C_0/|B|$.

On the other hand, we can see that $P_{m,j}\longrightarrow 0$, as $j\rightarrow \infty$, in $L^1(\mathbb{R}^n)$. To prove this, using again (\ref{M1}), for every $j\in \mathbb{N}$, we have that
\begin{align}\label{Qalpha}
P_{m,j}(x)&=(\tau _{x_0}D_{M_{x_0,r-j\ell}^{-1}}\Pi_{B,m}D_{M_{x_0,r-j\ell }}\tau _{-x_0})(b_j)(x)\nonumber\\
&=\sum_{\alpha \in \mathbb{N}^n,|\alpha |\leq m}\frac{1}{|B|}\int_B (D_{M_{x_0,r-j\ell }}\tau _{-x_0})(b_j)(z)z^\alpha dz (\tau _{x_0}D_{M_{x_0,r-j\ell }^{-1}})(Q_\alpha)(x)\nonumber\\
&=\sum_{\alpha \in \mathbb{N}^n,|\alpha |\leq m}\frac{1}{|B|}\int_B b_j(M_{x_0,r-j\ell }z+x_0)z^\alpha dz (\tau _{x_0}D_{M_{x_0,r-j\ell }^{-1}})(Q_\alpha)(x)\nonumber\\
&=\frac{1}{|B||E_j|}\sum_{\alpha \in \mathbb{N}^n,|\alpha |\leq m}\int_{E_j} b(z)(M_{x_0,r-j\ell }^{-1}(z-x_0))^\alpha dz \;Q_\alpha(M_{x_0,r-j\ell }^{-1}(x-x_0)),\quad x\in E_j.
\end{align}

Property $(M1)$ implies that, for every $\alpha \in \mathbb{N}^n$, $|\alpha |\leq m$,
\begin{equation}\label{M2}
\int_{E_j} b(z)(M_{x_0,r-j\ell }^{-1}(z-x_0))^\alpha dz =-\int_{E_j^c}b(z)(M_{x_0,r-j\ell }^{-1}(z-x_0))^\alpha dz,\quad j\in \mathbb{N}.
\end{equation}

Then, for every $j\in \mathbb{N}$, since
$$
\|Q_\alpha(M_{x_0,r-j\ell }^{-1}(\cdot-x_0))\|_{L^1(E_j)}=|E_j|\|Q_\alpha \|_{L^1(B)},\quad \alpha \in \mathbb{N}^n,|\alpha |\leq m,
$$
we obtain that
$$
\|P_{m,j}\|_{L^1(\mathbb{R}^n)}\leq \frac{1}{|B|}
\sum_{\alpha \in \mathbb{N}^n,|\alpha |\leq m}\left|\int_{E_j^c}b(z)(M_{x_0,r-j\ell }^{-1}(z-x_0))^\alpha dz\right|\|Q_\alpha \|_{L^1(B)}.
$$
Now, according to Proposition \ref{Prop2} it follows that
\begin{align*}
|(M_{x_0,r-j\ell }^{-1}(z-x_0))^\alpha|&\leq |M_{x_0,r-j\ell }^{-1}(z-x_0)|^{|\alpha|}\leq \|M_{x_0,r-j\ell }^{-1}\|^{|\alpha|}|z-x_0|^{|\alpha|}\\
&\leq C2^{\widetilde{\lambda}(r-j\ell)|\alpha|}|z-x_0|^{|\alpha |}\leq C|z-x_0|^{|\alpha |},\quad z\in \mathbb{R}^n,\,\alpha \in \mathbb{N}^n,|\alpha |\leq m,\,j\in \mathbb{N}.
\end{align*}

Thus, Proposition \ref{Prop4} allows us to deduce that
$$
\lim_{j\rightarrow \infty}\int_{E_j^c}b(z)\tau _{x_0}(D_{M_{x_0,r-j\ell}^{-1}} (z^\alpha ))dz=0,
$$
and consequently,
\begin{equation}\label{Pj}
\lim_{j\rightarrow \infty }\|P_{m,j}\|_{L^1(\mathbb{R}^n)}=0.
\end{equation}

Let us now consider the functions $W_{m,j}$, $j\in \mathbb{N}$, as follows:
$$
W_{m,j}=b_j-P_{m,j}.
$$

It can be seen that $W_{m,j}\longrightarrow b$, as $j\rightarrow \infty$, in $L^1(\mathbb{R}^n)$. Indeed, by taking into account Proposition \ref{Prop4} and that  $\cup_{j\in \mathbb{N}}E_j=\mathbb{R}^n$ and $E_j\subset E_{j+1}$, $j\in \mathbb{N}$, we deduce that
$$
\lim_{j\rightarrow +\infty}\|b_j-b\|_{L^1 (\mathbb{R}^n)}=\lim_{j\rightarrow \infty }\int_{E_j^c}|b(x)|dx=0,
$$
which jointly with (\ref{Pj}) leads to $\lim_{j\rightarrow \infty}W_{m,j}=b$ in $L^1(\mathbb{R}^n)$.

Thus, we can write
$$
b=W_{m,0}+\sum_{j\in \mathbb{N}} (W_{m,j+1}-W_{m,j}),
$$
where the series converges in $L^1(\mathbb{R}^n)$, and then also in $S'({\Bbb R}^n)$.

It is clear that $\int W_{m,j}(x)x^\alpha dx=0$, $\alpha \in \mathbb{N}^n$, $|\alpha |\leq m$, and $j\in \mathbb{N}$. Also, $W_{m,j}$, $j\in\mathbb{N}$, has compact support. Hence, $W_{m,0}$ and $W_{m,j+1}-W_{m,j}$, $j\in \mathbb{N}$, are multiple of $(p,q,m)$-atoms. We are going to show that
$$
W_{m,0}=\lambda _0{\mathfrak a}_0,
$$
and
$$
W_{m,j+1}-W_{m,j}=\lambda _j{\mathfrak a}_j,\quad j\in \mathbb{N},
$$
where $\{{\mathfrak a}_j\}_{j\in \mathbb{N}}$ is a sequence of $(p,q,m)$-atoms and $\{\lambda _j\}_{j\in \mathbb{N}}\subset (0,\infty)$ satisfies that $\sum_{j\in \mathbb{N}}\lambda_j^p\leq C$, for a certain $C>0$ that does not depend on $b$.

Note that by (\ref{Pi2}) we get
\begin{align*}
\|W_{m,0}\|_{L^q(\mathbb{R}^n)}&\leq \|b\|_{L^q(E_0)}+\|P_{m,0}\|_{L^q(E_0)}\leq  \|b\|_{L^q(E_0)}+C_1|E_0|^{-1/q'}\int_{E_0}|b(z)|dz\\
&\leq (1+C_1) \|b\|_{L^q(E_0)}\leq (1+C_1)\gamma ^{1/q-1/p}\leq (1+C_1)2^{-(r+1)(1/q-1/p)}.
\end{align*}
Then, by taking into account (C1), it follows that
$$
\|W_{m,0}\|_{L^q(\mathbb{R}^n)}\leq (1+C_1)(2a_2)^{1/p-1/q}|\theta (x_0,r)|^{1/q-1/p}.
$$
We define $\lambda _0=(1+C_1)(2a_2)^{1/p-1/q}$ and ${\mathfrak a}_0=\frac{1}{\lambda _0}W_{0,m}$. Thus it is clear that ${\mathfrak a}_0$ is a $(p,q,m)$-atom.

Let now $j\in \mathbb{N}$. We put
\begin{equation}\label{decomposition}
W_{m,j+1}-W_{m,j}=\chi _{E_{j+1}\setminus E_j}b-P_{m,j+1}+P_{m,j}.
\end{equation}

Firstly we estimate $\|\chi_{E_{j+1}\setminus E_j}b\|_{L^q(\mathbb{R}^n)}$. Let $c_0>1$ the constant which appears in Corollary \ref{coro}. By using (\ref{coro(b)}) and taking into account that $a_6\leq a_4$ ,  we get that,
\begin{align*}
\rho (x_0,z)&\geq a_12^{-\lambda ^{-1}(2\log_2c_0+\widetilde{\lambda}(r-j\ell))}
=a_1\left\{\begin{array}{ll}
2^{-(2\log_2c_0+a_6(r-j\ell))/a_4},& \displaystyle\;{\rm when }\;r-j\ell < -\frac{2\log_2c_0}{a_6},\\
2^{-(2\log_2c_0+a_6(r-j\ell))/a_6},& \;\displaystyle {\rm when }\;-\frac{2\log_2c_0}{a_6}\leq \;r-j\ell <0,\\
2^{-(2\log_2c_0+a_4(r-j\ell ))/a_6},&\;{\rm when }\;r-j\ell \geq 0,
\end{array}
\right.\\
&\geq C\left\{\begin{array}{ll}
2^{-(r-j\ell )a_6/a_4},& \;{\rm when }\;r-j\ell <0,\\
2^{-(r-j\ell )a_4/a_6},&\;{\rm when }\;r-j\ell \geq 0,
\end{array}
\right.
\end{align*}
provided that $z\in E_{j+1}\setminus E_j$. Thus, we conclude that
\begin{align}\label{chi}
\|\chi_{E_{j+1}\setminus E_j}b\|_{L^q(\mathbb{R}^n)}&= \left(\int_{E_{j+1}\setminus E_j}|b(z)\rho(x_0,z)^d|^q\frac{1}{\rho (x_0,z)^{dq}}dz\right)^{1/q}\nonumber\\
&\leq C\|b\rho (x_0,\cdot)\|_{L^q(\mathbb{R}^n)}\left\{\begin{array}{ll}
2^{d(r-j\ell )a_6/a_4},& \;{\rm when }\;r-j\ell <0,\\
2^{d(r-j\ell )a_4/a_6},&\;{\rm when }\;r-j\ell \geq 0.
\end{array}
\right.
\end{align}

Secondly we estimate $\|P_{m,j}\|_{L^q(\mathbb{R}^n)}$. By (\ref{Qalpha}) and (\ref{M2}) we get
\begin{align*}
\|P_{m,j}\|_{L^q(E_j)}&\leq \frac{1}{|B||E_j|^{1/q'}}\sum_{\alpha  \in  \mathbb{N}^n,|\alpha| \leq m}\left|\int_{E_j}b(z)(M_{x_0,r-j\ell}^{-1}(z-x_0))^\alpha dz\right|\|Q_\alpha \|_{L^q(B)}\\
&\leq \frac{C}{|E_j|^{1/q'}}\sum_{\alpha  \in  \mathbb{N}^n,|\alpha |\leq m}\left|\int_{E_j^c}b(z)(M_{x_0,r-j\ell}^{-1}(z-x_0))^\alpha dz\right|,
\end{align*}
where $C>0$ does not depend on $b$ nor on $j$.

Since $M_{x_0,r-j\ell }^{-1}(z-x_0)\not \in B$, for every $z\notin E_j$, we obtain
$$
\left|\int_{E_j^c}b(z)(M_{x_0,r-j\ell}^{-1}(z-x_0))^\alpha dz\right|\leq \int_{E_j^c}|b(z)||M_{x_0,r-j\ell}^{-1}(z-x_0)|^m dz,\quad \alpha \in \mathbb{N}^n,|\alpha |\leq m.
$$
By using Proposition \ref{Prop2} it follows that
$$
|M_{x_0,r-j\ell}^{-1}(z-x_0)|\leq \|M_{x_0,r-j\ell}^{-1}\||z-x_0|\leq C2^{\widetilde{\lambda}(r-j\ell )}|z-x_0|,\quad z\in \mathbb{R}^n.
$$
Then, we can write
\begin{align*}
\|P_{m,j}\|_{L^q(E_j)}&\leq C\frac{2^{m\widetilde{\lambda }(r-j\ell)}}{|E_j|^{1/q'}}\int_{E_j^c}|b(z)||z-x_0|^mdz\\
&\leq  C\frac{2^{m\widetilde{\lambda }(r-j\ell)}}{|E_j|^{1/q'}}\|b\rho (x_0,\cdot)\|_{L^q(\mathbb{R}^n)}\left(\int_{E_j^c}\frac{|z-x_0|^{mq'}}{\rho (x_0,z)^{dq'}}dz\right)^{1/q'},\\
&\leq  C\frac{2^{m\widetilde{\lambda }(r-j\ell)}}{|E_j|^{1/q'}}\|b\rho (x_0,\cdot)\|_{L^q(\mathbb{R}^n)}\left(\sum_{\beta \in \mathbb{N}}\int_{E_{j+\beta+1}\setminus E_{j+\beta }}\frac{|z-x_0|^{mq'}}{\rho (x_0,z)^{dq'}}dz\right)^{1/q'}.
\end{align*}

Corollary \ref{coro} allows us to see that, for every $\beta\in \mathbb{N}$ and $z\in E_{j+\beta +1}\setminus E_{j+\beta}$,
$$
\rho (x_0,z)\geq a_12^{-\lambda ^{-1}(2\log_2c_0+\widetilde{\lambda}(r-(j+\beta )\ell ))},
$$
and
$$
|z-x_0|\leq c_02^{-\lambda (r-(j+\beta+1)\ell )}.
$$

Thus we deduce that
\begin{align}\label{Abeta}
\|P_{m,j}\|_{L^q(E_j)}&\leq  C\frac{2^{m\widetilde{\lambda }(r-j\ell)}}{|E_j|^{1/q'}}\|b\rho (x_0,\cdot)\|_{L^q(\mathbb{R}^n)}\left(\sum_{\beta \in \mathbb{N}}\frac{2^{-mq'\lambda (r-(j+\beta +1)\ell)}}{2^{-dq'\lambda ^{-1}(2\log_2c_0+\widetilde{\lambda }(r-(j+\beta )\ell ))}}|E_{j+\beta+1}|\right)^{1/q'}\nonumber\\
&\leq C\|b\rho (x_0,\cdot)\|_{L^q(\mathbb{R}^n)}\left(\sum_{\beta \in \mathbb{N}}2^{mq'\widetilde{\lambda }(r-j\ell)-mq'\lambda (r-(j+\beta +1)\ell)+\beta \ell+dq'\lambda ^{-1}(2\log_2c_0+\widetilde{\lambda }(r-(j+\beta )\ell ))}\right)^{1/q'}\nonumber\\
&=: C\|b\rho (x_0,\cdot)\|_{L^q(\mathbb{R}^n)}\Big(\sum_{\beta \in \mathbb{N}}A_{r,j}(\beta)\Big)^{1/q'}.
\end{align}

Let us estimate $A_{r,j}(\beta)$, $\beta \in \mathbb{N}$. We distinguish three cases.
\begin{itemize}
\item Assume that $r-j\ell <-\frac{2\log_{2}c_0}{a_6}$. It is clear that, for every $\beta \in \mathbb{N}$, $r-(j+\beta)\ell <-\frac{2\log_{2}c_0}{a_6}$. Then, we get
\begin{align}\label{1}
A_{r,j}(\beta)&=2^{mq'a_6(r-j\ell)-mq'a_4(r-(j+\beta +1)\ell )+\beta \ell +\frac{dq'}{a_4}(2\log_2c_0+a_6(r-(j+\beta)\ell ))}\nonumber\\
&=2^{mq'a_4\ell +\frac{dq'}{a_4}2\log_2c_0}2^{(r-j\ell )[mq'(a_6-a_4)+\frac{dq'a_6}{a_4}]}
2^{\beta \ell(1+mq'a_4-\frac{dq'a_6}{a_4})},\quad \beta \in \mathbb{N}.
\end{align}

\item Suppose now that $-\frac{2\log_{2}c_0}{a_6} \leq r-j\ell<0$ and consider $\beta_0\in \mathbb{N}$ for which
\begin{equation}\label{beta0}
r-(j+\beta_0+1)\ell<\frac{-2\log_2c_0}{a_6}\leq r-(j+\beta_0)\ell.
\end{equation}
If $\beta \in \mathbb{N}$ and $\beta \geq \beta _0+1$, then $A_{r,j}(\beta )$ is as in (\ref{1}). In the case that $\beta \in \mathbb{N}$ and $\beta \leq \beta _0$, we get
$$
A_{r,j}(\beta)=2^{mq'a_4\ell +\frac{dq'}{a_6}2\log_2c_0}2^{(r-j\ell )[mq'(a_6-a_4)+dq']}2^{\beta \ell(1+mq'a_4-dq')}.
$$
Now, by taking into account that $0< a_6\leq a_4$ and $r-j\ell <0$, we deduce that
\begin{equation}\label{2}
A_{r,j}(\beta)\leq C2^{(r-j\ell )[mq'(a_6-a_4)+\frac{dq'a_6}{a_4}]}
2^{\beta \ell(1+mq'a_4-\frac{dq'a_6}{a_4})},\quad \beta\in \mathbb{N}.
\end{equation}
Here $C$ does not depend on $r$ nor on $j$.

\item In the end let us consider $r-j\ell \geq 0$. Let $\beta _1\in \mathbb{N}$ be such that
$$
r-(j+\beta_1+1)\ell<0\leq r-(j+\beta_1)\ell ,
$$
and $\beta _0\in \mathbb{N}$ verifying (\ref{beta0}). Note that $\beta_1\leq \beta_0$. We get
$$
A_{r,j}(\beta)= \left\{\begin{array}{ll}
\displaystyle 2^{mq'a_6\ell +\frac{dq'}{a_6}2\log_2c_0}2^{(r-j\ell )[mq'(a_4-a_6)+\frac{dq'a_4}{a_6}]}2^{\beta \ell(1+mq'a_6-\frac{dq'a_4}{a_6})},&\quad 0\leq \beta\leq \beta _1-1,\\
\displaystyle 2^{mq'a_4\ell+\frac{dq'}{a_6}2\log_2c_0}2^{(r-j\ell )\frac{dq'a_4}{a_6}}2^{\beta\ell(1+mq'a_4-\frac{dq'a_4}{a_6})},&\quad \beta =\beta _1,\\
\displaystyle 2^{mq'a_4\ell +\frac{dq'}{a_6}2\log_2c_0}2^{(r-j\ell )dq'}2^{\beta \ell(1+mq'a_4-dq')},&\quad \beta_1+1\leq \beta\leq \beta _0,\\
\displaystyle 2^{mq'a_4\ell +\frac{dq'}{a_4}2\log_2c_0}2^{(r-j\ell )\frac{dq'a_6}{a_4}}2^{\beta \ell(1+mq'a_4-\frac{dq'a_6}{a_4})},&\quad \beta\geq \beta _0+1.
		\end{array}
\right.
$$
Again, since $0<a_6\leq a_4$, we obtain, when $r-j\ell \geq 0$,
\begin{equation}\label{3}
A_{r,j}(\beta)\leq C2^{(r-j\ell )[mq'(a_4-a_6)+\frac{dq'a_4}{a_6}]}
2^{\beta \ell(1+mq'a_4-\frac{dq'a_6}{a_4})},\quad \beta\in \mathbb{N}.
\end{equation}

By considering (\ref{Abeta}), (\ref{2}) and (\ref{3}) we deduce that
\begin{align}\label{P}
\|P_{m,j}\|_{L^q(\mathbb{R}^n)}&\leq C\|b\rho (x_0,\cdot )\|_{L^q(\mathbb{R}^n)}\left(\sum_{\beta \in \mathbb{N}}2^{\beta \ell(1+mq'a_4-\frac{dq'a_6}{a_4})}\right)^{1/q'}\left\{\begin{array}{ll}
2^{(r-j\ell )[m(a_6-a_4)+\frac{da_6}{a_4}]},& r-j\ell <0,\\
2^{(r-j\ell )[m(a_4-a_6)+\frac{da_4}{a_6}]},&r-j\ell \geq 0.
\end{array}
\right.\nonumber\\
& \leq C\|b\rho (x_0,\cdot )\|_{L^q(\mathbb{R}^n)} \left\{\begin{array}{ll}
2^{(r-j\ell )[m(a_6-a_4)+\frac{da_6}{a_4}]},& \;{\rm when }\;r-j\ell <0,\\
2^{(r-j\ell )[m(a_4-a_6)+\frac{da_4}{a_6}]},&\;{\rm when }\;r-j\ell \geq 0,
\end{array}
\right.
\end{align}
because $d>a_4(1-1/q+ma_4)/a_6$.

\end{itemize}

And finally, by taking into account (\ref{decomposition}), (\ref{chi}) and (\ref{P}), we obtain that
$$
\|W_{m,j+1}-W_{m,j}\|_{L^q(\mathbb{R}^n)}\leq C\|b\rho (x_0,\cdot)\|_{L^q(\mathbb{R}^n)}
\left\{\begin{array}{ll}
2^{(r-j\ell )(m(a_6-a_4)+da_6/a_4)},&{\rm if }\;\;r-j\ell <0,\\
2^{(r-j\ell )(m(a_4-a_6)+da_4/a_6)},&{\rm if }\;\;r-j\ell \geq 0.
\end{array}
\right.
$$

Now, since $M_{p,q,m,d}(b)=1$,
\begin{align*}
\|b\rho (x_0,\cdot)^d\|_{L^q(\mathbb{R}^n)}&=\Big(\|b\|_{L^q(\mathbb{R}^n)}^{(1- \sigma)\alpha _1}+\|b\|_{L^q(\mathbb{R}^n)}^{(1- \sigma)\alpha _2}\Big)^{-1/\sigma}\\
&\leq \gamma ^{\frac{1-\sigma}{\sigma}(1/p-1/q)\alpha _i}\leq 2^{-r\frac{1-\sigma}{\sigma}(1/p-1/q)\alpha _i},\quad i=1,2.
\end{align*}

Then, when $r-j\ell <0$ we can write
\begin{align}\label{W1}
\|W_{m,j+1}-W_{m,j}\|_{L^q(\mathbb{R}^n)}&\leq C2^{-r\frac{1-\sigma}{\sigma}(1/p-1/q)\alpha _1+(r-j\ell )[m(a_6-a_4)+da_6/a_4+1/q-1/p]}|E_{j+1}|^{1/q-1/p}\nonumber\\
&\hspace{-1cm}\leq C2^{r[(1/q-1/p)(\frac{1-\sigma}{\sigma }\alpha _1+1)+m(a_6-a_4)+da_6/a_4]}2^{-j\ell (m(a_6-a_4)+da_6/a_4+1/q-1/p)}|E_{j+1}|^{1/q-1/p}\nonumber\\
&\hspace{-1cm}=C_12^{-j\ell (m(a_6-a_4)+da_6/a_4+1/q-1/p)}|E_{j+1}|^{1/q-1/p},
\end{align}
because $\alpha _1=\frac{1}{d(1-\sigma)}\Big(da_6/a_4-m(a_4-a_6)+1/q-1/p\Big)$.

In the same way we obtain that, when $r-j\ell \geq 0$,
\begin{align}\label{W2}
\|W_{m,j+1}-W_{m,j}\|_{L^q(\mathbb{R}^n)}&\leq C2^{-r\frac{1-\sigma}{\sigma}(1/p-1/q)\alpha _2+(r-j\ell )[m(a_4-a_6)+da_4/a_6+1/q-1/p]}|E_{j+1}|^{1/q-1/p}\nonumber\\
&\hspace{-1cm}\leq C2^{r[(1/q-1/p)(\frac{1-\sigma}{\sigma }\alpha _2+1)+m(a_4-a_6)+da_4/a_6]}2^{-j\ell (m(a_4-a_6)+da_4/a_6+1/q-1/p)}|E_{j+1}|^{1/q-1/p}\nonumber\\
&\hspace{-1cm}=C_22^{-j\ell (m(a_4-a_6)+da_4/a_6+1/q-1/p)}|E_{j+1}|^{1/q-1/p},
\end{align}
since $\alpha _2=\frac{1}{d(1-\sigma )}\Big(da_4/a_6+m(a_4-a_6)+1/q-1/p\Big)$ Note that the positive constants $C_1$ and $C_2$ are independent of $b$.

We define $\lambda_j$ as follows,
$$
\lambda _j=\left\{\begin{array}{ll}
C_12^{-j\ell (m(a_6-a_4)+da_6/a_4+1/q-1/p)},& {\rm when} \;\;r-j\ell <0,\\
C_22^{-j\ell (m(a_4-a_6)+da_4/a_6+1/q-1/p)},&{\rm when}\;\; r-j\ell \geq 0.
\end{array}
\right.
$$
We can write
$$
W_{j+1,m}-W_{j,m}=\lambda _j{\mathfrak a}_j,
$$
where ${\mathfrak a}_j=(W_{j+1,m}-W_{j,m})/\lambda _j$. Then, ${\mathfrak a}_j$ is a $(p,q,m)$-atom and, since $d>a_4(m(a_4-a_6)+1/p-1/q)/a_6$,
$$
\sum_{j\in \mathbb{N}}\lambda _j^p\leq C,
$$
where $C>0$ does not depend on $b$.

The proof of the theorem is thus finished.

\end{proof}

\end{document}